\documentclass[12pt]{article}
\usepackage[usenames, dvipsnames]{color}
\usepackage{authblk}
\usepackage{amsfonts}
\usepackage{mathrsfs}
\usepackage{amsmath}
\usepackage{amssymb,extarrows}
\usepackage{multicol}
\usepackage{float}
\usepackage{makeidx}
\usepackage{layout}
\usepackage{array}
\usepackage{a4wide}
\usepackage{boxedminipage}
\usepackage{hyperref}
\usepackage{latexsym}
\usepackage{color}
\usepackage{dsfont}
\usepackage{graphicx}
\usepackage{ulem}
\usepackage{amsthm}
\title{A note on the development of singularities on solutions to the Navier-Stokes equations under super critical forcing terms}
\author[1]{Hugo Beir\~{a}o da Veiga \thanks{Hugo Beir\~{a}o da Veiga (\texttt{hbeiraodaveiga@gmail.com}) is partially supported by FCT (Portugal) under the project: UIDB/MAT/04561/2020.}}
\affil[1]{\small{Department of Mathematics, Pisa University, Pisa, Italy\par
Academia das Ci\^encias de Lisboa (the Portuguese Academy of Science)}}
\author[2]{Jiaqi Yang\thanks{Jiaqi Yang (\texttt{yjqmath@nwpu.edu.cn, yjqmath@163.com}) is supported by NSF of China under Grant: 12001429.}}
\affil[2]{\small{School of Mathematics and Statistics, Northwestern Polytechnical University, Xi'an, 710129, China}}
\date{}

\newtheorem{theorem}{Theorem}[section]

\newtheorem{lemma}[theorem]{Lemma}
\newtheorem{corollary}[theorem]{Corollary}

\newtheorem{problem-2}[theorem]{Problem}
\theoremstyle{remark}
\newtheorem{remark}{Remark}[section]

\theoremstyle{definition}

\numberwithin{equation}{section}

\newcommand{\Om}{{\Omega}}

\newcommand{\ep}{{\epsilon}}
\newcommand{\cp}{{\widetilde{p}}}
\newcommand{\cq}{{\widetilde{q}}}

\newcommand{\p}{\partial}

\newcommand{\R}{\mathbb{R}}

\newcommand{\al}{\alpha}
\newcommand{\f}{\frac}
\newcommand{\n}{\nabla}

\newcommand{\ed}{{\end{document}}}

\begin{document}
\maketitle
\begin{abstract}
Recently Qi S. Zhang provides examples of solutions to the Navier-Stokes equations which, under suitable hypothesis, blow up in finite time. He considers axially symmetric solutions in a cylinder $D\,$ under appropriate boundary conditions and under the effect of super critical external forces $f\,.$ The loss of boundedness for the velocity field, as $t\rightarrow T\,,$ is the basic case of blow up. However a more general situation is considered below, as explained in the preamble.\par
In his main result Zhang exhibits, for each $q<\infty\,,$ a blow up solution with an external force  $f\in L^q(0,T;L^1(D))\,.$ Following Zhang, we construct blow up solutions with forcing terms in $L^q(0,T;L^p(D))\,,$ for suitable pairs $(q,p)\,.$ In particular our results contain Zhang's result. A significant particular case is the existence of external forces $f \in L^1(0,T;L^p(D))\,,$ for every $p<2$, for which the velocity blows up in a finite time. The significant case $p=2$ remains open.

\end{abstract}
%
\section{Preamble}

The blow-up problem for solutions to the Navier-Stokes equations, even under suitable external forces, is certainly one of the most important open problems in the related theory. Curiously, as far as we know, the corresponding literature is very poor.\par
To prove the global regularity in time of solutions under the effect of any bounded external force or, conversely, deny this statement, that is, to exhibit bounded forces under which the solutions blows-up, seems to us a goal of truly exceptional scope, beyond our claims. Let's recall that the fundamental, very profound, known results of non-uniqueness of the solution must take into account non-regular external forces.\par
In the following, by blow up as $t\rightarrow T\,,$ we mainly mean loss of boundedness for the velocity field. But blow up of higher order derivatives will be also considered, since their boundedness is necessary if one wants to ensure that solutions remain classical, as they are here before the blow up time. In fact, a particularly simple, and clarifying feature of our approaches is that external forces, and solutions, are smooth in $[0,T)$. Blow up occurs just as $t\rightarrow T\,.$\par

Let's show a very particular case. Given $p<2$, we show the existence of bounded, smooth, external forces in the time interval $[0,T[$ which belong to the space $L^1(0,T;L^p(\Om))$, and such that the corresponding velocities, obviously regular in $[0,T[$, blows-up at time $T$. Unfortunately, the case $p=2$ remains open.

\vspace{0.2cm}

It may appear negative to some readers that our solution also solves the linear Stokes equations under the same external force. We want here to partially refute this conclusion.\par
Let us assume that the same result was instead proved under a strong and fundamental appeal to the non-linear term. We ask ourselves how this result would then have been considered. By taking into account the lack of better results in the literature known to us, we are led to believe that the answer would have been favorable to the result. This belief leads us to show the proof of the present result, whose greatest merit is maybe its simplicity. We hope that future research with recourse to the nonlinear term will lead to the same conclusion but under the effect of "more integrable" external forces.
\section{Introduction}
In this paper, we are concerned with the Navier-Stokes equations
\begin{equation}\label{NS}
	\begin{cases}
		\Delta v-v\cdot\n v-\n P-\p_tv=f&\quad\text{in $D\times(0,\infty)$}\,,\\
		\n\cdot v=0&\quad\text{in $D\times(0,\infty)$}\,,\\
		v(x,0)=v_0\,,
	\end{cases}
\end{equation}
under suitable boundary conditions, where $v$ is the velocity, $P$ is the pressure, and $f$ is a given forcing term. The well known classical spaces $L^q(0,T;L^p(D))$ are simply denoted by $L^q_tL^p_x\,.$

J. Leray \cite{leray} proved that for divergence-free $v_0\in L^2$ and $f\in L^1_tL^2_x $, there exists a global weak solution to the above problem, now known as Leray-Hopf solutions (see \cite{hopf} for the contribution of Hopf in bounded domains). Existence of global regular solutions and uniqueness remain the more important open problem in the theory of the Navier-Stokes equations, at least from the theoretical point of view. Below, by looking for blow up solutions, we are indirectly concerned with the above first problem (however we give some information on the second problem).

After Leray and Hopf's works many results on the above subjects appeared. An important progress is given by the so called Ladyzhenskaya-Prodi-Serrin sufficient condition for regularity:
If $u\in L^q(0,T;L^p(\Omega))$ with
\begin{equation}\label{1p2}
	\f2q+\f3p=1\,,\quad 3< p\leq+\infty\,,
\end{equation}
then $u$ is regular. See \cite{giga}, \cite{B87}, and \cite{G-P}, where completely different and independent proofs are shown (concerning \cite{B87} we refer to the 7th. section in reference \cite{BY20}). Assumption \eqref{1p2} was previously considered by other authors in studying the uniqueness problem. See Galdi's reference \cite{Ga} for a very interesting treatment of the classical results on energy equality, uniqueness, and regularity. Further, the first author of this paper extend the above result to the case of assumptions on the vorticity $\omega=\n\times u$, see \cite{BV-95}. It is proved that if $\omega\in L^q(0,T;L^p(\Omega))$ with
\begin{equation}
	\f2q+\f3p=2\,,\quad \f32\leq p\leq+\infty\,,
\end{equation}
then $u$ is regular.\par
The much harder case $p=3$ in equation \eqref{1p2} was reached by L. Escauriaza, G.A. Seregin, V. \v{S}ver\'{a}k, see \cite{ESS}.

\vspace{0.2cm}

Concerning non-uniqueness of \eqref{NS}, not treated below, a fundamental contribution was given in O.A. Ladyzhenskaya reference \cite{La-69}, following a previous idea due to V. Scheffer. However the spatial domain depends on time (roughly, like a revolution cone with the time direction as rotation axis, and the vertex at time zero). Nevertheless the result is significant, since for regular data the solution is unique.\par
In \cite{BV} non-uniqueness was proven for some solutions in the space $L^{\infty}_tL^2_x\,,$ for $f=0$. On the other hand, D. Albritton, E. Bru\'e and M. Colombo established, in reference \cite{ABC}, the non-uniqueness of \eqref{NS} in the energy space for a forcing term in the super-critical space $L^1_tL^2_x$. Recall that a forcing term in local $L^q_tL^p_x$ spaces is super-critical, critical, or sub-critical if $\f2q+\f3p$ is larger, equal, or smaller than $2,$ respectively.

\section{The blow up problem}
\subsection{Leray's proposal}
In a famous 1934 work (see \cite{leray}, chapter III, paragraph 20, page 224) Leray proposed the construction of possible self-similar solutions of the Navier-Stokes equations developing a singularity at a certain time $t=T$. Leray looks for self-similar solutions of the following form:
\begin{equation}
u=\f{1}{\sqrt{2(T-t)}}U\left(\f{x}{\sqrt{2(T-t)}}\right)\,,\quad P=\frac{1}{2(T-t)} \Pi(\frac{x}{\sqrt{(T-t)}})\,,
\end{equation}
where $(U(y),\Pi(y))$ should satisfy
\begin{equation}
\begin{split}
-\Delta U+U+(y\cdot\n) U+(U\cdot\n)U+\n\Pi=0\,,\\
\text{div} U=0\,.
\end{split}
\end{equation}
If $U\neq0$ then $u$ develops a singularity at $T$. A main difficulty in Leray's approach is due to avoiding external forces. In fact, in the absence of external forces, even in the most general case, the non-existence of Leray's blow up self-similar solutions was proved in 1996, after 62 years of attempts, by J. Ne\v{c}as, M. R$\mathring{\text{u}}$zi\v{c}ka, and V. \v{S}ver\'{a}k in the historical reference \cite{NRS}.\par

However, if we introduce a self-similar external force $f$, namely
\begin{equation}
f=\frac{1}{[2(T-t)]^{3/2}} F\left(\frac{x}{\sqrt{2(T-t)}}\right)\,, 
\end{equation}
the situation becomes much simpler. In this case, $(U(y),\Pi(y))$ should satisfy
\begin{equation}
	\begin{split}
		-\Delta U+U+(y\cdot\n) U+(U\cdot\n)U+\n\Pi=F\,,\\
		\text{div} U=0\,.
	\end{split}
\end{equation}
If $U=0$, then $\n\Pi=F$ which implies that $\n\times F=0$\,. The following conclusion holds: \\
{\it{If $\n\times F\neq0$, then $U\neq0$.}}\\
Hence, the external force will lead to a singularity.

\subsection{Zhang's approach}
One may found the very basic approach underlying Zhang's article in the above famous Leray's 1934 work. In comparison to Leray's proposal, Zhang's setting is extremely simplified. On the other hand restrictions made more difficult the real, effective, existence of blow up solutions. However, by admitting suitable external forces, the existence of significant blow up solutions holds.\par%
It is worth noting that, as far as we know, there are no proof of blow up under bounded external forces (similarly to non-uniqueness).\par
We may merely quote the 1985 reference \cite{Sch-85} where V. Scheffer considered weak solutions to the Navier-Stokes equations with an external force that always acts against the direction of the flow, and proved that there exists a solution with an internal singularity. However, as the author remarks, $f$ is very singular, and $u$ is not a real solution. See \cite{Sch-85} assertion 2.

\vspace{0.2cm}

Let's now give some useful insight into the proof of Zhang's main theorem, which establish the following result, see \cite{Zhang} Theorem 1.1-(2): For any $q>1$ there is a regular solution of \eqref{NS} in $[0,T[\times D\,,$ with a forcing term in the super-critical space $L^q_tL^1_x\,,$ which blows up at the final time $T\,.$\par
The author starts from the following axially symmetric ($\p_\theta v=0$) Navier-Stokes equations under the Navier (total) slip boundary condition:
\begin{equation}\label{ASNS}
	\begin{cases}
		\left(\Delta-\f1{r^2}\right)v_r-(v_r\p_r+v_3\p_{x_3})v_r+\f{v^2_{\theta}}{r}-\p_rP-\p_tv_r=f_r&\quad\text{in $D\times(0,\infty)$}\,,\\
		\left(\Delta-\f1{r^2}\right)v_{\theta}-(v_r\p_r+v_3\p_{x_3})v_{\theta}+\f{v_{\theta}v_r}{r}-\p_rP-\p_tv_{\theta}=f_{\theta}&\quad\text{in $D\times(0,\infty)$}\,,\\
		\Delta v_3-(v_r\p_r+v_3\p_{x_3})v_{3}-\p_{x_3}P-\p_tv_{x_3}=f_{3}&\quad\text{in $D\times(0,\infty)$}\,,\\
		\f1r\p_r(rv_r)+\p_{x_3}v_3=0&\quad\text{in $D\times(0,\infty)$}\,,\\
		\p_{x_3}v_r=\p_{x_3}v_{\theta}=0\,,v_3=0& \quad\text{on $\p^HD$}\,,\\
		\p_{x_3}v_{\theta}=\f{v_{\theta}}{r}\,,\,\p_rv_{x_3}=0\,, v_r=0&\quad\text{on $\p^VD$}\,,
	\end{cases}
\end{equation}
where
\[v_r=v\cdot e_r\,,\, v_\theta=v\cdot e_\theta\,,\,v_3=v\cdot e_3\,,\]
and
\[f_r=v\cdot f_r\,,\, f_\theta=f\cdot e_\theta\,,\,f_3=f\cdot e_3\,.\]
$\p^HD$ and $\p^VD$ denote respectively the horizontal and vertical parts of the boundary $\p D$. Concerning results on the axially symmetric Navier-Stokes equations, classical references are, among others, \cite{La} and \cite{UY}. It is known that if $v_{\theta}= 0$ and $f$ is regular, then blow up does not occur in $\R^3\,$ for any time $t>0$. Hence it looks suitable to start by the more stringent case, namely by assuming in \eqref{ASNS} not merely that $v_{\theta}\neq 0,$ but also that $v_r=v_3=0$, $f_r=f_3=0$, $f_{\theta}=f_{\theta}(r,t)\,.$ The problem reduces to solve the following system:
\begin{equation}\label{reduce1}
	\begin{cases}
		\f{v^2_{\theta}}{r}-\p_rP=0\,,\\
		\left(\Delta-\f{1}{r^2}\right)v_{\theta}-\p_tv_{\theta}=f_{\theta}\,,\\
		\p_{x_3}P=0\,.
	\end{cases}
\end{equation}

Let $\phi$ be a solution of the equation
\begin{equation}\label{reduce2}
	\begin{cases}
		\left(\p_{r}^2-\f{1}{r^2}\right)\phi-\p_t\phi=h\quad\text{in $D\times[0,T]$}\,,\\
		\phi(0,t)=0\,,
	\end{cases}
\end{equation}
where $v_{\theta}=\phi$, $h=f_{\theta}$, and
\[P=\int_0^r\f{v^2_{\theta}}{l}\,dl
\,.\]
Then the couple $v_\theta\,,P$ satisfy \eqref{reduce1}\,.

Next Zhang looks for self similar solution of \eqref{reduce2} in the form of
\begin{equation}
	\phi=\f{1}{\sqrt{2(T-t)}}\phi_0\left(\f{r}{\sqrt{2(T-t)}}\right)\,,\,
	h=\f{1}{[2(T-t)]^{\f32}}k\left(\f{r}{\sqrt{2(T-t)}}\right)\,.
\end{equation}
Then \eqref{reduce2} reduces to the following ODE:
\begin{equation}\label{reduce3}
	\begin{cases}
		\phi''_0+\f1r\phi'_0-\f1{r^2}\phi_0-\phi_0-r\phi'_0=k(r)\,,\\
		\phi_0(0)=0\,,
	\end{cases}
\end{equation}
where $k=k(r)\,,$ $r\in [0,\infty)\,,$ is a smooth function supported in the unit interval $[0,1]$.\par
Zhang obtained an explicit solution of \eqref{reduce3}, namely
\[
\phi_0=-\f{1}{r}\int_0^rse^{\f{s^2}{2}}\int_0^\infty e^{-\f12l^2}k(l)dlds\,.
\]
Starting from this explicit solution, by setting  $\phi=\f{1}{\sqrt{2(T-t)}}\phi_0\left(\f{r}{\sqrt{2(T-t)}}\right)$, Zhang constructed the desired blow-up solution by taking $v=[\ln(\phi(r,t)+1)-\ln(\phi(1,t)+1)r]\,e_{\theta}$, which is no more self-similar.\par
\begin{remark}\label{main}
It is worth noting that the assumption $v_r=v_3=0$ implies that $v\cdot\n v=(v_r\p_r+v_3\p_{x_3})v=0$. Hence the nonlinear term in the Navier-Stokes equation has essentially no effect. It follows that proofs and results also apply to the Stokes evolution problem. It follows that under bounded external forces blow-up is here not possible. However this fact is not significant by itself since, as far as we know, better results are not disposable, even under the help of the non linear term. Please, see the last part of the preamble.\par
A clarifying feature of our blow up result is that external forces and solutions are smooth in $D\times[0,T)$. Blow up occurs only as $t\rightarrow T$.\par
To end, note that the situation for the Stokes evolution problem is a generalization of that for real functions: Blow-up of $u'(t)$ does not imply blow-up of $u(t).$
\end{remark}
Concerning related recent results, for the reader's convenience we repeat the references quoted in \cite{Zhang}, namely \cite{CSTY08,CSTY09,CFZ,KNSGS,LPYZZZ,Pan}\,. In \cite{LPYZZZ}, the authors constructed a finite time blow up example for a special cusp type domain.

\section{Motivation, strategy, main results, and remarks}

 \subsection{Motivation} Since in the classical literature weak solutions are generally studied under square-integrable forces with respect to spatial variables, a natural question is to ask whether the solution of \eqref{NS} may blow up for a forcing term in $L^1_tL^2_x$, usually considered in classical treatises like J. Leray's reference \cite{leray} and O.A. Ladyzhenskaya reference \cite{La-book69}.
 This motivation was the starting point of the present paper. Below we obtain external forces in $L^1_tL^p_x$, for each $p<2$. However, the classical case $L^1_tL^2_x$ remains open.

 \subsection{Strategy}
  We appeal to Zhang's basic ideas, by starting from his explicit solution $\phi$. But we obtain a supplementary degree of freedom $\al$ essentially by multiplication of the previous solutions by $r^\al$. More precisely, we consider solutions of the form $v=[r^\al\phi(r,t)-\phi(1,t)r]e_{\theta}$\,. Compared with Zhang's paper \cite{Zhang}, our advantage is that we may adjust $\al$ to obtain a distinct, much larger, set of blow-up solutions. Clearly we loose the self-similar structure of the very basic solutions. Note that our parameter $\al$ is not related to the $\al$ in Zhang's reference, introduced to satisfy the boundary condition on the lateral surface of the cylinder.

\vspace{0.2cm}

\subsection{Main Results}
In the sequel $D$ denotes the cylinder $B_2(0,1)\times[0,1]$ in $\R^3$, where $B_2(0,1)$ is the unit ball in $\R^2$ centred at origin. Everywhere	we assume that $\,1\leq q\leq \infty\,,$  $\,1\leq p < \infty\,,$ and  $1\leq k \leq 3\,.$ Hence $0\leq \alpha <3\,.$ Occasional different situations would be explicitly remarked.
\begin{theorem}\label{maintheo}
	Let $\al$ and a couple of exponents $q,p$ be given, such that $k-1\leq \al< k$ for some integer $k\,,$ $1\leq k \leq 3\,.$ and such that the constraint
	\begin{equation}\label{3menos}
		[(3-\al)p-2]q<2
	\end{equation}
holds. Then there exists a force $f\in L^q_tL^p_x$ and a corresponding, unique, solution $ v \in L^\infty_t L^2_x \cap L^2_t H^1_x $ of the Navier-Stokes equations
	\begin{equation}\label{NSEQ}
		\Delta v-v\cdot\n v-\n p-\p_tv=f
	\end{equation}
	in $D\times[0,T)$ which satisfies the non slip boundary condition $v=0$ on the vertical boundary of $D$ and the Navier total slip boundary condition on the horizontal boundary of $D$. In addition, $v$ satisfies the point wise blow up relation
\begin{equation}\label{geral}
\|\n^{k-1}v(x,t)\|_{L^{\infty}_x}\to +\infty \quad \textrm{as} \quad t\to T\,.
\end{equation}
Moreover $v\in L^{\cq}_tL^{\cp}_x$ for all pair $\cp,\cq$ satisfying the inequality
\begin{equation}\label{cpq}
		[(1-\al)\cp-2]\cq<2\,.
\end{equation}
External forces $f$ and solutions $v$ are smooth in $D\times[0,T)$.\\
If $k\geq2$ then one also has $v\in C^{k-2}(D\times[0,T])$.\par
\end{theorem}
\begin{remark}\label{clasic}
Note that for $k=2$ equation \eqref{geral} means blow up of the vorticity, and for $k=3$ means blow up of the solutions, as classical solutions. For $k\geq 2\,,$ equation \eqref{cpq} loses interest since it is trivially verified for any pair of positive reals. This is a forewarning that the solution will be much smoother, as claimed in the theorem.
\end{remark}

Let's now put in evidence the case case $k=1$. Note that assumption \eqref{3menos} implies that $q<\infty$.
\begin{corollary}\label{coro}
For any $0\leq \al< 1$, and any couple $q,p$ satisfying the constraint \eqref{3menos}, there exists a force $f\in L^q_tL^p_x$ such that the corresponding, unique, solution $v$ of the Navier-Stokes equations \ref{NSEQ} satisfying the no slip boundary condition $v=0$ on the vertical boundary of $D$ and
the Navier total slip boundary condition on the horizontal boundary of $D\,,$ satisfies the blow up condition
	\begin{equation}\label{estco}
		\lim_{ t\to T} \|v(x,t)\|_{L^{\infty}_x}= +\infty\,.
	\end{equation}
Moreover $v\in L^{\cq}_tL^{\cp}_x$ for all couple $\cp, \cq$ such that $[(1-\al)\cp-2]\cq<2$.\par
External forces $f$ and solutions $v$ are smooth in $D\times[0,T)$. Blow up occurs only at time $T$.\par
\end{corollary}
The uniqueness claim follows from the smoothness for $t<T$.\par
The particular case below appears clarifying. For simplicity we assume that the exponents $\cp$ and $\cq$ are a priori given.
\begin{corollary}\label{maincoro}
Let be given a pair of exponents $(q,p)$such that  $p<1+\f1q\,,$
and a couple of arbitrary large exponents $\cq$ and $\cp\,.$  Then there is a force $f\in L^q_tL^p_x$ such that the corresponding solution $v$ satisfies the general properties described in Corollary \ref{coro}, belongs to the space $ L^{\cq}_tL^{\cp}_x\,,$ and verifies the unboundedness property \eqref{estco}.
\end{corollary}
For convenience we prove here this last result. Fix $\ep>0$ such that $(1+\ep) p= 1+\f1q\,,$ and $\ep \cp\leq 1+\f{1}{\cq}\,.$ Next fix $\al$ such that $1-2\ep<\al<1\,.$ It easily follows that \eqref{3menos} holds. Hence there is an external force $f=f_\ep \in L^{q}_tL^{p}_x$ for which \eqref{estco} holds. Furthermore $[(1-\al)\cp-2]\cq< [2\ep\cp-2]\cq<2\,.$ Hence
$v\in L^{\cq}_tL^{\cp}_x\,.$
\begin{remark}
For any arbitrarily large $q$, by taking $p=1\,,$ we obtain \cite[Theorem 1.1, item (2)]{Zhang}.
\end{remark}
\begin{remark}
By taking $q=1$ and $p<2\,,$ we show that there are external forces $f \in L^1_tL^p_x\,$ such that \eqref{estco} holds for all $p<2$. For $p=2$ the problem remains open.\par
\end{remark}
\begin{remark}
In correspondence to any couple of fixed, but arbitrarily large, exponents $\cq$ and $\cp$, there are blow up solutions satisfying $v \in L^\cq_t L^\cp_x\,.$ Note that solutions are not merely strong solutions up to time $T$, but even much more. In fact, the classical quantity $2/\cq + 3/\cp\,,$ always positive here, can be chosen arbitrarily small. For instance, in correspondence to any fixed, arbitrarily large exponent $\cp$, there are blow up solutions $v$ satisfying $v\in L^\infty_t L^\cp_x\,.$ But the border line case $\cp= \infty $ can not occur, due to our blow up result. So, the result looks significant.
\end{remark}

\vspace{0.2cm}

For $p=2$, there exist external forces $f\in  L^q_tL^2_x\,,$ for arbitrarily large $q\,,$ for which
\begin{equation}\label{nab}
	\|\n v(x,t)\|_{L^\infty_x}\to\infty,\quad t\to T\,.
\end{equation}
On the other hand, there are external forces $f \in L^1_tL^p_x\,,$ for arbitrarily large $p\,$ for which \eqref{nab} holds.\par
It could be of some interest to consider the case $q=p\,.$ The above corollary shows that blow up holds if $f \in L^q_tL^p_x\,,$ with $q=p<\f{1+\sqrt{5}}{2}\,.$\par

A last remark: It is well known that, roughly, if  $f \in L^q_tL^p_x\,,$ with $\f2q + \f3p <2$ then solutions of Stokes linear problem are bounded (actually, H\"older continuous). As the reader easily verifies, in the above three particular cases token into consideration, namely, the two extreme cases $q=1$ and $q=\infty$, and the one with $q=p$, the quantity  $\f2q + \f3p$ is larger then $2$. Actually, one easily verifies that there is a positive constant $c$ such that $\f2q + \f3p > 2+c\,,$ for all pairs $(q,p)$ leading to the blow up results described in the corollary.
\section{Proof of Theorem \ref{maintheo}}
First, by Zhang's paper (1.11), it follows that  $\phi=\f{1}{\sqrt{2(T-t)}}\phi_0\left(\f{r}{\sqrt{2(T-t)}}\right)$ solves the equation
\begin{equation}
	\left(\Delta-\f{1}{r^2}\right)\phi-\p_t\phi=h\,,
\end{equation}
where $r=|x|$,
\begin{equation}
	\phi_0=\phi_0(r)=-\f{1}{r}\int_0^rse^{\f{s^2}{2}}\int_0^\infty e^{-\f12l^2}k(l)dlds\,,
\end{equation}
\begin{equation}
	h=\f{1}{[2(T-t)]^{\f32}}k\left(\f{r}{\sqrt{2(T-t)}}\right)\,,
\end{equation}
and $k=k(r)$  is a fixed smooth function, defined for all $r>0\,,$ but supported in the unit interval  $[0,1]$.

Set $\tilde{\phi}=r^{\al}\phi$. One has
\begin{equation}
	\left(\Delta-\f{1}{r^2}\right)\tilde{\phi}-\p_t\tilde{\phi}=r^{\al}h+\al^2r^{\al-2}\phi+2\al r^{\al-1}\p_r\phi:=h_1\,.
\end{equation}
We remark that $\tilde{\phi}$ and $h_1$ will be, respectively, the solution $v$ and the data $f$ in equation \eqref{NSEQ}.\par
It is easy to check that (see, Zhang'paper (1.16))
\begin{equation}
	|\phi|\leq\f{Cr}{r^2+T-t}\,,\quad |\p_r\phi|\leq\f{C}{r^2+T-t}\,,
\end{equation}
so we have that
\begin{equation}
	|\al^2r^{\al-2}\phi+2\al r^{\al-1}\p_r\phi|\leq C\f{r^{\al-1}}{r^2+T-t}\,.
\end{equation}
\begin{lemma}\label{lem}
	For all $1+p\al-p>-1$, that is $\al>1-\f{2}{p}$, we have that
	\begin{equation}
		\int_D|h_1|^pdx\leq
		\begin{cases}
			\f{C}{(T-t)^{\f{3p-2-p\al}{2}}} \quad \textrm{if}&\,\f{3p-p\al}{2}\neq1\,,\\
			C|\ln(T-t)|\quad \textrm{if}&\quad \f{3p-p\al}{2}=1\,.
		\end{cases}
	\end{equation}
\end{lemma}
\begin{proof}
	First, we have that
	\begin{equation}
		\begin{split}
			\int_{D}|\al^2r^{\al-2}\phi+2\al r^{\al-1}\p_r\phi|^{p}dx
			\leq C\int_0^1\f{r^{1+p\al-p}}{(r^2+T-t)^{p}}dr\,.
		\end{split}
	\end{equation}
	
	When $1-\f{2}{p}<\al\leq1$, we split the interval $[0,1]$ into two parts: $[0,\sqrt{T-t}]$ and $[\sqrt{T-t},1]$.
	
	In $[0,\sqrt{T-t}]$, we have that
	\begin{equation}
		\begin{split}
			\int_0^{\sqrt{T-t}}\f{r^{1+p\al-p}}{(r^2+T-t)^{p}}dr
			\leq\f{1}{(T-t)^p}\int_0^{\sqrt{T-t}}r^{1+p\al-p}dr
			\leq\f{C}{(T-t)^{\f{3p-2-p\al}{2}}}\,.
		\end{split}
	\end{equation}
	
	In $[\sqrt{T-t},1]$, since $p-p\al\geq0$, we have that
	\[\f{(r^2+T-t)^{\f{p-p\al}{2}}}{r^{p-p\al}}=
	\left(\f{r^2+T-t}{r^2}\right)^{\f{p-p\al}{2}}\leq C\]
	for $r\in[\sqrt{T-t},1]$. Thus, we have
	\begin{equation}
		\begin{split}
			&\int_{\sqrt{T-t}}^1\f{r^{1+p\al-p}}{(r^2+T-t)^{p}}dr\\
			&=\int_{\sqrt{T-t}}^1\f{r}{(r^2+T-t)^{\f{3p-p\al}{2}}}\f{(r^2+T-t)^{\f{p-p\al}{2}}}{r^{p-p\al}}dr\\
			&\leq C\int_{\sqrt{T-t}}^1\f{r}{(r^2+T-t)^{\f{3p-p\al}{2}}}dr\\
			&=\f{C}2\int_{\sqrt{T-t}}^1\f{1}{(r^2+T-t)^{\f{3p-p\al}{2}}}d(r^2+T-t)\\
			&=C\begin{cases}
				\f1{3p-2-p\al}\left[\f{1}{[2(T-t)]^{\f{3p-2-p\al}{2}}}-\f{1}{(1+T-t)^{\f{3p-2-p\al}{2}}}\right]\,,\quad\f{3p-p\al}{2}\neq1\,,\\
				\f12[\ln(1+T-t)-\ln(T-t)]\,,\quad \f{3p-p\al}{2}=1\,,
			\end{cases}\\
			&=	\begin{cases}
				\f{C}{(T-t)^{\f{3p-2-p\al}{2}}}\,,\quad\f{3p-p\al}{2}\neq1\,,\\
				C|\ln(T-t)|\,,\quad \f{3p-p\al}{2}=1\,.
			\end{cases}
		\end{split}
	\end{equation}

\vspace{0.2cm}
	
	When $\al>1$, since $p\al-p>0$, it follows that
	\[\f{r^{p\al-p}}{(r^2+T-t)^{\f{p\al-p}{2}}}=
	\left(\f{r^2}{r^2+T-t}\right)^{\f{p\al-p}{2}}\leq1\,,\]
	for $r\in[0,1]$, which gives
	\begin{equation}
		\begin{split}
			&\int_{D}|\al^2r^{\al-2}\phi+2\al r^{\al-1}\p_r\phi|^{p}dx\\
			&\leq C\int_0^1\f{r^{1+p\al-p}}{(r^2+T-t)^{p}}dr\\
			&= C\int_0^1\f{r}{(r^2+T-t)^{\f{3p-p\al}{2}}}\f{r^{p\al-p}}{(r^2+T-t)^{\f{p\al-p}{2}}}dr\\
			&\leq C\int_0^1\f{r}{(r^2+T-t)^{\f{3p-p\al}{2}}}dr\\
			&\leq	\begin{cases}
				\f{C}{(T-t)^{\f{3p-2-p\al}{2}}}\,,\quad\f{3p-p\al}{2}\neq1\,,\\
				C|\ln(T-t)|\,,\quad \f{3p-p\al}{2}=1\,.
			\end{cases}
		\end{split}
	\end{equation}

\vspace{0.2cm}
	
	On the other hand, we have for all $\al\geq0$
	\begin{equation}
		\begin{split}
			&\int_{D}|r^{\al}h|^{p}dx\\
			&\leq \int_0^1\f{r^{p\al+1}}{[2(T-t)]^{\f{3p-1}{2}}}k^p\left(\f{r}{\sqrt{2(T-t)}}\right)d\left(\f{r}{\sqrt{2(T-t)}}\right)\\
			&=\f{1}{[2(T-t)]^{\f{(3-\al)p-2}{2}}}\int_0^1\left(\f{r}{\sqrt{2(T-t)}}\right)^{\f{3p-1}{2}}k^p\left(\f{r}{\sqrt{2(T-t)}}\right)d\left(\f{r}{\sqrt{2(T-t)}}\right)\,.	
		\end{split}
	\end{equation}
\end{proof}

From \eqref{3menos} we obtain
\begin{equation}
	p\al-p+1>2p-\f2q-1\geq-1\,.
\end{equation}
Thus one shows that $(3p-2-p\al)q<2$ implies $p\al-p+1>-1$\,. Then, from Lemma \ref{lem}, it follows that
\begin{equation}
	h_1\in L^q_tL^p_x
\end{equation}
for all $p, q,$ and $\alpha$ such that
\begin{equation}
	(3p-2-p\al)q<2\,.
\end{equation}

\vspace{1 cm}

Next we study the solution $v= \tilde{\phi}$. One has
\begin{equation}
	|\tilde{\phi}|\leq C\f{r^{\al+1}}{r^2+T-t}\,,\quad|\p_r\tilde{\phi}|\leq C\f{r^{\al}}{r^2+T-t}\,.
\end{equation}

Hence
\begin{equation}
	\int_D|\tilde{\phi}|^2dx\leq C\int_0^1\f{r^{2\al+3}}{(r^2+T-t)^2}dr\leq C\,,
\end{equation}
and also
\begin{equation}
	\begin{split}
		&\int_D|\p_r\tilde{\phi}|^2dx\\
		&\leq C\int_0^1\f{r^{2\al+1}}{(r^2+T-t)^2}dr\\
		&\leq C\int_0^1\f{r}{(r^2+T-t)^{2-\al}}\f{r^{2\al}}{(r^2+T-t)^{\al}}dr\\
		&\leq C\int_0^1\f{r}{(r^2+T-t)^{2-\al}}dr\\
		&\leq
		\begin{cases}
			\f{C}{(T-t)^{1-\al}}\,,\quad\al<1\,,\\
			C|\ln(T-t)|\,,\quad \al=1\,,\\
			C(1+T-t)^{\al-1}\,,\quad\al>1 \,.
		\end{cases}
	\end{split}
\end{equation}
If $\al>0$, then $ \f{C}{(T-t)^{1-\al}}\in L^1_t$, so $\tilde{\phi}\in L^\infty_tL_x^2\cap L^2_tH^1_x$ for $\al>0$.

Recall that
\begin{equation}
	\tilde{\phi}=r^{\al}\phi=\f{r^{\al}}{\sqrt{2(T-t)}}\phi_0\left(\f{r}{\sqrt{2(T-t)}}\right)\,.
\end{equation}
One has
\begin{equation}
	\begin{split}
		\p^k\tilde{\phi}=&\sum_{i=0}^kC^i_k(\p^{i}r^{\al})\p^{k-i}\phi\\
		\sim \,& \f{r^{\al}}{(\sqrt{2(T-t)})^{k+1}}(\p^k\phi_0)\left(\f{r}{\sqrt{2(T-t)}}\right)
		+\f{C(\al,k)r^{\al-k}}{\sqrt{2(T-t)}}\phi_0\left(\f{r}{\sqrt{2(T-t)}}\right)\,.
	\end{split}
\end{equation}

Note that $\phi_0$ has support in $[0,1]$, hence we may assume that
\[\f{r}{\sqrt{2(T-t)}}\leq1\,.\]
So if $k-1\leq \al< k$, one has
\begin{equation}
	\f{r^{\al}}{(\sqrt{2(T-t)})^{(k-2)+1}}\leq\left(\f{r}{\sqrt{2(T-t)}}\right)^{k-1}r^{\al-(k-1)}\leq 1\,,
\end{equation}
and
\begin{equation}
	\f{r^{\al-(k-2)}}{\sqrt{2(T-t)}}=\f{r}{\sqrt{2(T-t)}}r^{\al-(k-1)}\leq 1\,,
\end{equation}
which gives that
\begin{equation}
	\tilde{\phi}\in C^{k-2}(D\times[0,T])\,.
\end{equation}
On the other hand, if $k-1\leq \al< k$, one has
\begin{equation}
	\f{r^{\al}}{(\sqrt{2(T-t)})^{(k-1)+1}}=\left(\f{r}{\sqrt{2(T-t)}}\right)^{\al}\left(\f{1}{\sqrt{2(T-t)}}\right)^{k-\al}\to\infty\,,  \text{as $t\to T$}\,,
\end{equation}	
which shows that
\begin{equation}
	\p^{k-1}_r\tilde{\phi}(x,t)\to\infty,\quad \text{as $t\to T$}\,.
\end{equation}

One has
\begin{equation}\label{below}
	\begin{split}
		\int_D|\tilde{\phi}|^{\cp}dx\leq&C\int_0^1\f{r^{\cp\al+\cp+1}}{(r^2+T-t)^\cp}dr\\
		=&C\int_0^1\f{r}{(r^2+T-t)^{\cp-\f{\cp\al+\cp}{2}}}
		\f{r^{\cp\al+\cp}}{(r^2+T-t)^{\f{\cp\al+\cp}{2}}}dr\\
		\leq& C\int_0^1\f{r}{(r^2+T-t)^{\cp-\f{\cp\al+\cp}{2}}}dr\\
		\leq&\begin{cases}
			\f{C}{(T-t)^{\f{\cp-2-\cp\al}{2}}}\,,\quad\f{\cp-\cp\al}{2}\neq1\,,\\
			C|\ln(T-t)|\,,\quad \f{\cp-\cp\al}{2}=1\,.
		\end{cases}\,.
	\end{split}
\end{equation}

Hence $\tilde{\phi}\in L^\cq_tL^\cp_x$ provided that $[(1-\al)\cp-2]\cq<2$.

Finally, to verify that the boundary conditions are satisfied, set $v=(\tilde{\phi}+\beta r)e_{\theta}$, where $\beta=\int_0^1se^{\f{s^2}{2}}\int_s^\infty e^{-\f{l^2}{2}}k(l)dlds$. Actually, the Navier total slip boundary condition on the horizontal boundary of $D$ is satisfied since $v$ is independent of the vertical variable $x_3$ and $v_r=v_3=0$. On the other hand, it follows from Zhang's equation (1.20) that
\begin{equation}
	\tilde{\phi}(1,t)=-\int_0^1se^{\f{s^2}{2}}\int_s^\infty e^{-\f{l^2}{2}}k(l)dlds=-\beta\,.
\end{equation}
So the no slip boundary condition is satisfied on the vertical boundary. The proof of Theorem \ref{maintheo} is accomplished.

\end{document}